\documentclass[11pt,a4paper]{amsart}

\usepackage{lmodern,microtype,mathtools,booktabs,verbatim}
\usepackage[dvipsnames,svgnames,x11names,hyperref]{xcolor}
\usepackage{amsmath,amsthm,amssymb,fancyhdr,graphicx,bbm,cancel,color,mathrsfs,todonotes, pinlabel}
\usepackage[pagebackref,colorlinks,citecolor=Mahogany,linkcolor=Mahogany,urlcolor=Mahogany,filecolor=Mahogany]{hyperref}

\usepackage{bm}

\newtheorem{thm}{Theorem}
\newtheorem{mainthm}{Theorem}

\newtheorem{lem}[thm]{Lemma}
\newtheorem{prop}[thm]{Proposition}

\theoremstyle{definition} 
\newtheorem{deftn}[thm]{Definition}
 
\theoremstyle{remark}
\newtheorem{example}[thm]{Example}

\newtheorem{rmk}[thm]{Remark}

\newcommand{\dmo}{\DeclareMathOperator}

\newcommand{\R}{\mathbb{R}}
\newcommand{\Q}{\mathbb{Q}}
\newcommand{\co}{\mathbb{C}}\newcommand{\Z}{\mathbb{Z}}
\newcommand{\A}{\mathbb{A}}

\newcommand{\De}{\Delta}
\newcommand{\Ga}{\Gamma}
\newcommand{\wtil}{\widetilde}

\newcommand{\sbs}{\subset}\newcommand{\bs}{\backslash}

\newcommand{\xra}{\xrightarrow}

\newcommand{\ra}{\rightarrow}
\newcommand{\hra}{\hookrightarrow}
\newcommand{\bb}[1]{\mathbb{#1}}\newcommand{\ca}[1]{\mathcal{#1}}\newcommand{\mf}{\mathfrak}

\newcommand{\fr}[2]{\frac{#1}{#2}}

\newcommand{\ti}{\times}

\newcommand{\lra}{\longrightarrow}

\dmo{\sgn}{sign}
\dmo{\we}{\wedge}
\dmo{\ind}{ind}\dmo{\Ind}{Ind}
\dmo{\bop}{\bigoplus}\dmo{\pic}{Pic}
\dmo{\coker}{coker}\dmo{\vol}{Vol}\dmo{\gal}{Gal}\dmo{\perm}{Perm}
\dmo{\tor}{Tor}\dmo{\ext}{Ext}\dmo{\Ext}{Ext}
\dmo{\aut}{aut}
\dmo{\Aut}{Aut}
\dmo{\inn}{Inn}\dmo{\var}{Var}
\dmo{\dep}{depth}
\dmo{\ad}{ad}\dmo{\curl}{curl}
\dmo{\hy}{\bb H}\dmo{\Sl}{SL}
\dmo{\SO}{SO}\dmo{\psl}{PSL}
\dmo{\isom}{Isom}\dmo{\Isom}{Isom}
\dmo{\conf}{Conf}
\dmo{\stab}{Stab}\dmo{\Jac}{Jac }
\dmo{\diam}{diam}\dmo{\fix}{Fixed}\dmo{\Fix}{Fix}
\dmo{\injR}{injRad}\dmo{\Ad}{Ad}
\dmo{\esv}{ess-vol}\dmo{\out}{Out}\dmo{\Out}{Out}
\dmo{\nil}{Nil}\dmo{\sol}{Sol}
\dmo{\Div}{div}
\dmo{\SU}{SU}
\dmo{\SP}{SP}
\dmo{\Sp}{Sp}
\dmo{\SL}{SL}
\dmo{\rk}{rk}
\dmo{\rank}{rank}
\dmo{\psp}{PSp}\dmo{\psu}{PSU}
\dmo{\PU}{PU}\dmo{\pgl}{PGL}
\dmo{\Mod}{Mod}\dmo{\range}{Range}
\dmo{\eu}{eu}\dmo{\mi}{mi}
\dmo{\Log}{Log}\dmo{\supp}{supp}
\dmo{\maps}{Maps}\dmo{\Gr}{Gr}
\dmo{\Pin}{Pin}
\dmo{\Spin}{Spin}\dmo{\Str}{Str}
\dmo{\Sq}{Sq}\dmo{\Symp}{Symp}
\dmo{\pd}{PD}\dmo{\PD}{PD}\dmo{\sig}{Sig}
\dmo{\Set}{Set}\dmo{\Top}{Top}
\dmo{\ev}{ev}\dmo{\St}{St}
\dmo{\Pt}{Pt}\dmo{\pt}{pt}

\dmo{\colim}{colim }\dmo{\Pl}{PL}
\dmo{\String}{String}\dmo{\smear}{smear}

\dmo{\dev}{dev}
\dmo{\met}{Met}\dmo{\contact}{Contact}
\dmo{\teich}{Teich}\dmo{\Teich}{Teich}\dmo{\qi}{QI}
\dmo{\der}{Der}
\dmo{\cl}{Cliff}\dmo{\Cl}{Cl}
\dmo{\Pf}{Pf}
\dmo{\ch}{ch}\dmo{\diag}{diag}
\dmo{\grad}{grad}\dmo{\Char}{char}
\dmo{\spec}{Spec}\dmo{\Arg}{Arg}
\dmo{\rad}{rad}\dmo{\im}{Im}
\dmo{\Hom}{Hom}\dmo{\End}{End}
\dmo{\tr}{tr}\dmo{\id}{Id}
\dmo{\gl}{GL}
\dmo{\sym}{Sym}\dmo{\Sym}{Sym}
\dmo{\com}{Comm}
\dmo{\Lk}{Lk}
\dmo{\CAT}{CAT}
\dmo{\Rep}{Rep}
\dmo{\Conf}{Conf}
\dmo{\PConf}{PConf}
\dmo{\Push}{Push}
\dmo{\Cont}{Cont}
\dmo{\sm}{\setminus}
\dmo{\vn}{\varnothing}
\dmo{\disk}{\mathbb D}
\dmo{\Trd}{Trd}\dmo{\Mat}{Mat}
\dmo{\Riem}{Riem}
\dmo{\Diffn}{\Diff_0}\dmo{\diff}{diff}
\dmo{\Diff}{Diff}\dmo{\homeo}{Homeo}
\dmo{\Homeo}{Homeo}\dmo{\Fr}{Fr}
\dmo{\rot}{rot}\dmo{\Emb}{Emb}
\dmo{\Ham}{Ham}\dmo{\Met}{Met}
\dmo{\Ein}{Ein}\dmo{\CP}{\co P}
\dmo{\Per}{Per}\dmo{\Ric}{Ric}
\newcommand{\C}{\mathbb C}\dmo{\Nrd}{Nrd}
\dmo{\Comp}{Comp}\dmo{\PSC}{PSC}
\dmo{\Cent}{Cent}\dmo{\Orb}{Orb}
\dmo{\aind}{a-ind}\dmo{\tind}{t-ind}
\dmo{\constant}{constant}
\dmo{\Td}{Td}
\dmo{\LMod}{LMod}
\dmo{\SMod}{SMod}
\dmo{\SDiff}{SDiff}
\dmo{\Br}{Br}
\dmo{\csch}{csch}
\dmo{\triv}{triv}
\dmo{\genus}{genus}
\dmo{\Homeq}{HomEq}
\dmo{\PP}{\mathbb{P}}
\dmo{\U}{U}
\dmo{\Eis}{Eis}
\dmo{\Gal}{Gal}
\dmo{\BDiff}{\wtil{\Diff}}
\dmo{\BAut}{\wtil{\Aut}}
\dmo{\Iso}{Iso}
\dmo{\codim}{codim}
\dmo{\II}{II}
\dmo{\I}{I}
\dmo{\GL}{GL}
\dmo{\ph}{ph}
\dmo{\Stab}{Stab}
\dmo{\OO}{O}

\usepackage{enumerate,stmaryrd} 
\usepackage[margin=0.95in]{geometry}

\usepackage{tikz,tikz-cd}
\usetikzlibrary{matrix}

\begin{document}

\title[Characteristic classes of bundles of K3 manifolds]{Characteristic classes of bundles of K3 manifolds and the
  Nielsen Realization problem}

\author{Jeffrey Giansiracusa}
\address{Department of Mathematics, Swansea University, Singleton Park, Swansea SA2 8PP, UK} \email{j.h.giansiracusa@swansea.ac.uk}

\author{Alexander Kupers}
\address{Department of Mathematics, Harvard University, Cambridge, MA 02138} \email{kupers@math.harvard.edu}

\author{Bena Tshishiku}
\address{Department of Mathematics, Harvard University, Cambridge, MA 02138} \email{tshishikub@gmail.com}


\date{\today}


\begin{abstract} 
Let $K$ be the K3 manifold. In this note, we discuss two
methods to prove that certain generalized Miller--Morita--Mumford classes for smooth bundles with fiber $K$ are non-zero. As a consequence, we fill a gap in a paper of the first author, and prove that the homomorphism $\Diff(K)\ra\pi_0\Diff(K)$ does not split. One of the two methods of proof uses a result of Franke on the stable cohomology of arithmetic groups that strengthens work of Borel, and may be of independent interest.
\end{abstract}

\maketitle

\section{Introduction}

In this paper $K$ denotes the \emph{$K3$ manifold}, which is the underlying oriented manifold of a complex $K3$ surface. This uniquely specifies its diffeomorphism type, and one may construct it as the hypersurface in $\mathbb{C} P^3$ cut out by the homogeneous equation $z_0^4+z_1^4+z_2^4+z_3^4 = 0$. For each element $c \in H^i(B\SO(4);\Q)$, there is a characteristic class $\kappa_c$ of smooth oriented manifold bundles with fiber $K$, called a \emph{generalized Miller--Morita--Mumford class}: given such a bundle $E \to B$ we take the vertical tangent bundle $T_v E$ and integrate the class $c(T_v E) \in H^i(E;\Q)$ over the fibers to get $\kappa_c(E) \in H^{i-4}(B;\Q)$.

Let $\Diff(K)$ denote the group of orientation-preserving $C^2$-diffeomorphisms, in the $C^2$-topology. Its classifying space $B\Diff(K)$ carries a universal smooth manifold bundle with fiber $K$, and hence there are classes $\kappa_c \in H^*(B\Diff(K);\Q)$ which may or may not be zero. Letting $\mathcal{L}_2 = \frac{1}{45}(7 p_2 -p_1^2)$ denote the second Hirzebruch $L$-polynomial, we prove the following: 

\begin{mainthm}\label{thm:hirzebruch} The generalized Miller--Morita--Mumford-class $\kappa_{\mathcal{L}_2} \in H^4(B\Diff(K);\Q)$ is nonzero. 
\end{mainthm}

The Hirzebruch L-polynomials are related to signatures of
manifolds and as a corollary of Theorem \ref{thm:hirzebruch}, there
exists a smooth bundle of $K3$ manifolds over a closed stably-framed
4-manifold whose total space has nonzero signature. We shall give two
proofs of Theorem \ref{thm:hirzebruch}: the first is an explicit
calculation for the tautological bundle over a certain moduli space of $K3$ surfaces, while the second combines the study of Einstein metrics with a general result about cohomology of arithmetic groups following work of Franke.

Either proof can be combined with the Bott vanishing theorem to prove the following result. We define the mapping class group $\Mod(K)$ to be the group $\pi_0\Diff(K)$ of path components of $\Diff(K)$.

\begin{mainthm}\label{thm:nielsen}
The surjection  $p \colon \Diff(K) \to \Mod(K)$ does not split, i.e.\ there is no homomorphism $s \colon \Mod(K)\ra\Diff(K)$ so that $p\circ s=\id$. 
\end{mainthm}

This is an instance of the Nielsen realization problem; see e.g.\ \cite{mann-tshishiku}. 
Theorem \ref{thm:nielsen} first appeared in \cite{giansiracusa}, but the proof was flawed (see \cite{giansiracusa-corrigendum}). However, it can be repaired with small modifications and many of the ideas in this paper derive from \cite{giansiracusa}.

{\it Acknowledgements.} The authors thank H.\ Grobner and M.\ Krannich, as well as the anonymous referees, for helpful comments. The third author thanks B.\ Farb for introducing him to the paper \cite{giansiracusa}. 

\section{Quasi-polarized $K3$ surfaces} Suppose that $\pi \colon E \to B$ is an oriented manifold bundle with closed fibers of dimension $d$. This has a vertical tangent bundle $T_v E$ with corresponding characteristic classes $c(T_v E) \in H^i(E;\Q)$ for each $c \in H^i(B\SO(d);\Q)$. The generalized Miller--Morita--Mumford classes are obtained by integration of these classes along the fibers:
\[\kappa_c(E) \coloneqq \int_{\pi} c(T_v E) \in H^{i-d}(B;\Q).\]
Applying this construction to the universal bundle of $K3$ manifolds over $B\Diff(K)$ results in classes $\kappa_c \in H^{i-4}(B\Diff(K);\Q)$ for each $c \in H^{i}(B\SO(4);\Q) = \Q[e,p_1]$. 

These classes are natural in the bundle: for any continuous map $f \colon B' \to B$, $\kappa_c(f^* E) = f^* \kappa_c(E)$. To prove $\kappa_c \neq 0 \in H^*(B\Diff(K);\Q)$, it therefore suffices to find a single bundle $E \to B$ such that $\kappa_c(E) \neq 0$.

We shall use the moduli space $\ca M_{2d}$ of quasi-polarized $K3$ surfaces of degree $2d$ (the value of $d$ plays no role in our arguments). This is actually a stack with finite
automorphism groups of bounded order, but since we are interested in its rational cohomology we may ignore these technical details. We shall not go into
the details of its construction, but recall some facts from
\cite{vanderGeerKatsura,PetersenK3}. There is a universal family $\pi
\colon \ca X_{2d} \to \ca M_{2d}$ of $K3$ surfaces. As this is a bundle of complex surfaces, its vertical tangent bundle has Chern classes $t_i \coloneqq c_i(T_v \ca X_{2d}) \in H^{2i}(\ca X_{2d};\Q)$. The class $t_1$ is the pullback of a class $\lambda \in H^2(\ca M_{2d};\Q)$. The main result of \cite{vanderGeerKatsura} is that $\lambda^{17} \neq 0$ but $\lambda^{18} = 0$, in the Chow ring of $\ca M_{2d}$. Petersen gives the corresponding result in rational cohomology \cite{PetersenK3}, and attributes it to van der Geer and Katsura. We shall use this to prove the following improvement of Theorem \ref{thm:hirzebruch}:

\begin{prop}\label{prop:hirzebruchimprov} The generalized Miller--Morita--Mumford-class $\kappa_{\ca L_{i+1}} \in H^{4i}(B\Diff(K);\Q)$ is non-zero for $i \leq 8$.
\end{prop}

\begin{proof}It suffices to prove that $\kappa_{\ca L_{i+1}}(\ca
  X_{2d}) \neq 0$. Since the $K3$ manifold is 4-dimensional, $p_1$, $p_2$ are the only non-zero Pontryagin classes of the vertical tangent bundle. These can be expressed in terms of the Chern classes using \cite[Corollary 15.5]{MilnorStasheff}: 
	\[p_1(T_v \ca X_{2d}) = t_1^2-t_2 \quad \text{and} \quad p_2(T_v \ca X_{2d}) = t_2^2.\]
	We substitute these into the first nine Hirzebruch $L$-polynomials, as computed by McTague \cite{McTague}. Since integration along fibers is linear, it suffices to compute $\int_{\pi} t_1^i t_2^j$. As $t_1 = \pi^* \lambda$, the push-pull formula gives $\lambda^i \int_{\pi} t_2^j$, and \cite[Section 3]{vanderGeerKatsura} used Grothendieck--Riemann--Roch to determine that $\int_{\pi} t_2^j = a_{j-1} \lambda^{2j-2}$ for particular integers $a_{j-1}$. Using this, we compute that $\kappa_{\ca L_{i+1}}(\ca X_{2d})$ is a non-zero multiple of $\lambda^{2i}$ for $1 \leq i \leq 8$ and hence non-zero, cf.~Table \ref{tab:quasi}.\end{proof}

\begin{example}
	Let us do the computation for $i=3$ as an example:
	\[\ca L_4 = \frac{-19p_2^2+22p_1^2p_2-3p_1^4}{14175} \qquad \text{ ignoring terms involving $p_i$ with $i \geq 3$}\]
	\[\ca L_4(T_v \ca X_{2d}) = \frac{-3 t_1^8 + 24 t_1^6 t_2 - 50 t_1^4 t_2^2 + 8 t_1^2 t_2^3 + 21 t_2^4}{14175}\]
	\[\kappa_{\ca L_{i+1}}(\ca X_{2d}) = \int_{\pi} \ca L_4(T_v \ca X_{2d})  = \frac{24\lambda^6\cdot 24-50\lambda^4 \cdot 88 \lambda^2+8\lambda^2 \cdot184 \lambda^4+21 \cdot 352 \lambda^6}{14175} = \frac{16 \lambda^6}{45}.\qedhere \] 
\end{example}	

\begin{table}
	\centering
	\caption{The class $\kappa_{\ca L_{i+1}}(\ca X_{2d})$ in terms of the class $\lambda$ for $i \leq 8$.}
	\label{tab:quasi}
	\begin{tabular}{lccccccccc}
		\toprule
		$i$ & 0    & 1          & 2 & 3     & 4 & 5 & 6 & 7  &8   \\ \midrule
		$\kappa_{\ca L_{i+1}}(\ca X_{2d})$ & 24 & $8 \lambda^2$ & $\frac{8 \lambda^4}{3}$ &  $\frac{16 \lambda^6}{45}$ & $\frac{8 \lambda^8}{315}$ & $\frac{16 \lambda^{10}}{14175}$ & $\frac{16 \lambda^{12}}{467775}$ & $\frac{32 \lambda^{14}}{42567525}$ & $\frac{8 \lambda^{16}}{638512875}$ \\ \bottomrule
	\end{tabular}
\end{table}

\begin{rmk}The classes $\kappa_{\ca L_{i+1}}$ remain non-zero when pulled back to $H^{4i}(B\Diff(K\text{ rel }\ast);\Q)$, because $H^*(B\Diff(K),\Q) \to H^*(B\Diff(K \text{ rel } \ast),\Q)$ is injective: its composition with the Becker--Gottlieb transfer is given by multiplication with $\chi(K) = 24$. We do not know whether $\kappa_{\ca L_{i+1}}$ remains non-zero when pulled back to $H^{4i}(B\Diff(K\text{ rel }D^4);\Q)$.\end{rmk}

\section{Miller--Morita--Mumford classes and the action on homology}  \label{sec:mmm} One can also approach the group of diffeomorphisms of $K$ through its action on $H_2(K;\Z)$. In particular, we shall explain a relationship between the generalized Miller--Morita--Mumford classes and the arithmetic part of the mapping class group.

The middle-dimensional homology group $H_2(K;\Z) \cong \Z^{22}$ has intersection form given by $M = H \oplus H \oplus H \oplus -E_8 \oplus -E_8$, with $H$ the hyperbolic form and $-E_8$ the negative of the $E_8$-form. This is equivalent over $\R$ to the symmetric $(22 \times 22)$-matrix
\[B=\left(\begin{array}{ccc}
I_3&0\\
0&-I_{19}
\end{array}\right),\]
where $I_{n}$ is the $(n \ti n)$ identity matrix. In particular, we can consider $\Aut(M)$ as a subgroup of the Lie group $\OO(3,19)$.

The action of $\Mod(K)$ on $H_2(K;\Z)$ preserves the intersection form and hence induces a homomorphism $\alpha \colon \Mod(K)\ra \Aut(M)$, whose image $\Gamma_K$ is the index 2 subgroup of $\Aut(M)$ of those elements such that the product of the determinant and the spinor norm equals 1, cf.~\cite[\S 4.1]{giansiracusa}. 

The generalized Miller--Morita--Mumford classes associated to the Hirzebruch $L$-polynomials $\ca L_i \in H^{4i}(B\SO;\Q)$, whose pullback to $H^{4i}(B\SO(4);\Q)$ we shall denote in the same manner, can be obtained from the arithmetic group $\Ga_K$. We will now justify this claim.

There are homomorphisms 
\[\Ga_K \lra \Aut(M) \lra \OO(3,19) \overset{\simeq}\longleftarrow \OO(3) \times \OO(19).\]
Thus we get, up to homotopy, a map $w \colon B\Ga_K \lra B\OO(3) \times B\OO(19)$ which classifies a bundle $\eta$ with fibers $M \otimes \R$, which decomposes as a direct sum $\eta_+ \oplus \eta_-$ of a 3- and a 19-dimensional subbundle. We define a class 
\[x_{4i} \coloneqq  w^*(\ph_{4i} \otimes 1-1 \otimes \ph_{4i}) \in H^{4i}(B\Ga_K;\Q),\]
where $\ph_{4i}$ denotes the degree $4i$ component of the Pontryagin character.

By definition $x_{4i}$ is pulled back from $B\OO(3) \times B\OO(19)$, but it is in fact pulled back from $B\OO(3)$ \cite[Proposition 2.2]{giansiracusa}. By Chern--Weil theory the Pontryagin classes of the flat bundle $\eta$ vanish \cite[Corollary C.2]{MilnorStasheff}. This implies $\ph(\eta_+)+\ph(\eta_-)=0$, and thus $x_{4i} = \ph_{4i}(\eta_+)-\ph_{4i}(\eta_-) = 2\ph_{4i}(\eta_+)$, which is evidently pulled back along
\[B\Ga_K \lra B\OO(3) \times B\OO(19) \overset{\pi_1}{\lra} B\OO(3).\]

\begin{lem}\label{lem:atiyah} The pullback of $x_{4i} \in H^{4i}(B\Gamma_K;\Q)$ along the map $B\Diff(K) \to B\Gamma_K$ is equal to $1/2^{i+1} \kappa_{\ca L_{i+1}} \in H^{4i}(B\Diff(K);\Q)$.\end{lem}

\begin{proof}Atiyah proved that $x_{4i} \in H^{4i}(B\Ga_K;\Q)$ pulls back to $\smash{\kappa_{\tilde{\ca L}_{i+1}}} \in H^{4i}(B\Diff(K);\Q)$ along the map $B\Diff(K) \to B\Gamma_K$ \cite[\S4]{atiyah-signature}. Here $\tilde{\ca L}_{i+1}$ is the Atiyah--Singer modification of the Hirzebruch $L$-polynomials: while the latter has generating series $\sqrt{z}/\tanh(\sqrt{z})$, this modification has generating series $\sqrt{z}/\tanh(\sqrt{z}/2)$, so $2^{i+1} \tilde{\ca L}_{i+1} = \ca L_{i+1}$.\end{proof}

Let $\Ga_{\Ein} < \Ga_K$ be the index 2 subgroup of those elements such that both the determinant and the spinor norm are $1$; it has index 4 in $\Aut(M)$ and is the maximal subgroup contained in the identity component of $\OO(3,19)$. Restricting the previous maps to the identity component $\SO_0(3,19)$ in $\OO(3,19)$, we get
\[B\Ga_{\Ein} \lra B\SO_0(3,19) \overset{\simeq}{\longleftarrow} B\SO(3) \times B\SO(19) \overset{\pi_1}\lra B\SO(3).\]

To understand the induced map $H^*(B\SO(3);\Q) \to H^*(B\Ga_{\Ein};\Q)$, we introduce the space 
\[X_u = \frac{\SO(22)}{\SO(3) \times \SO(19)}.\]
In Section \ref{sec:franke} we shall discuss the Matsushima homomorphism
\[\mu \colon H^*(X_u;\C) \lra H^*(B\Ga_{\Ein};\C).\]
The principal $\SO(3)$-bundle $\SO(22)/\SO(19) \to X_u$ is classified by a 39-connected map $X_u \to B\SO(3)$ that factors over the map $X_u \to B\SO(3) \times B\SO(19)$. By \cite[Lemma 3.4]{giansiracusa} (a special case of \cite[Proposition 7.2]{BorelZeta}), the Matsushima homomorphism fits in a commutative diagram
\begin{equation}\label{eqn:borel-zeta-comm} \begin{tikzcd}  H^*(B\SO(3) \times B\SO(19);\C) \rar & H^*(X_u;\C) \arrow{d}{\mu}\\
H^*(B\SO(3);\C) \uar \rar & H^*(B\Gamma_{\Ein} ;\C).\end{tikzcd}\end{equation}
Changing coefficients to the complex numbers and pulling back $x_{4i}$ from $B\Gamma_K$ to $B\Gamma_{\Ein}$, we get $x_{4i} \in H^{4i}(B\Gamma_{\Ein};\C)$. From this we will conclude: 

\begin{lem}\label{lem:x-image}The class $x_{4i} \in H^{4i}(B\Gamma_{\Ein};\C)$ is in the image of the Matsushima homomorphism.\end{lem}

\begin{proof}The argument preceding Lemma \ref{lem:atiyah} tells us that in the commutative diagram
	\[\begin{tikzcd} H^*(B\SO(3);\C) \rar &  H^*(B\Gamma_{\Ein};\C) \\
	H^*(B\OO(3);\C) \rar \uar & H^*(B\Gamma_K;\C) \uar,\end{tikzcd}\]
	the element $x_{4i} \in H^*(B\Gamma_K;\C)$ is pulled back from $B\OO(3)$, and hence $x_{4i} \in H^*(B\Gamma_{\Ein};\C)$ is pulled back from $B\SO(3)$. The results then follows from the commutative diagram \eqref{eqn:borel-zeta-comm}.
\end{proof}

\section{Results of Franke and Grobner}\label{sec:franke}  In this section we explain a result about the Matsushima homomorphism, which implies:

\begin{prop}\label{prop:frankeEin}
	The homomorphism $H^*(B\SO(3);\C)\to H^*(X_u;\C) \to H^*(B\Ga_{\Ein};\C)$ is injective in degrees $*\le 20$. 
\end{prop}

Let $G$ be a connected semi-simple linear algebraic group over $\Q$. The real points $G(\R)$ form a semi-simple Lie group. Fix maximal compact subgroups $K<G(\R)$ and $U<G(\C)$ with $K \subset U$, let $Y_\infty \coloneqq G(\R) \slash K$ be the symmetric space of $G$, and $X_u \coloneqq U/K$ be the compact dual symmetric space of $G$. Fixing an arithmetic lattice $\Ga<G(\Q)$, by work of Matsushima and Borel \cite{matsushima,borel_cohoarith} there is a homomorphism $H^*(X_u;\C) \ra H^*(\Ga \backslash Y_\infty;\C)$ constructed using differential forms. Since $\Ga$ acts on the contractible space $Y_\infty$ with finite stabilizers, $H^*(\Ga \backslash Y_\infty;\C) \cong H^*(B\Ga;\C)$. We shall call the composition  
\begin{equation}\label{eqn:matsushima}\mu \colon H^*(X_u;\C)\lra H^*(\Ga \backslash Y_\infty;\C)\cong H^*(B\Ga;\C)\end{equation}
the \emph{Matsushima homomorphism}. It may be helpful to point out that $\mu$ in general is \emph{not} induced by a map of spaces, since it does not preserve the rational cohomology as a subset of the complex cohomology \cite{BorelZeta,okun}.

\begin{example}The Matsushima homomorphism discussed in the previous section is a particular instance of this. In this case $G = \SO(3,19)$, yielding $X_u$ as in the previous section. In this particular instance $\mu$ does preserve the rational cohomology in the range $* \leq 39$, as a consequence of the commutative diagram \ref{eqn:borel-zeta-comm}.\end{example}

Borel \cite{borel_cohoarith} proved that the Matsushima homomorphism is an isomorphism in a range of degrees, and by work of Franke \cite{franke} it is injective in a larger range.

\begin{thm}[Franke]\label{thm:franke}
	The homomorphism $(\ref{eqn:matsushima})$ is injective in degrees 
	\begin{equation}\label{eqn:unipotent}*\le \min_{R}\dim N_R,\end{equation} where $R$ ranges over maximal parabolic subgroups of $G$ over $\Q$, and $N_R\sbs R$ is the unipotent radical. 
\end{thm}

This is not stated explicitly in \cite{franke}, but a similar statement is given in \cite{grobner}, as we now explain. We require the following additional setup (see \cite{FrankeSchwermer}, \cite{li-schwermer}, \cite{speh-venkataramana} or \cite[\S6,8]{harder} for more information). Define the \emph{adelic symmetric space} $Y^\A$ and the \emph{adelic locally symmetric space} $X^\A$ by  
\[Y^\A \coloneqq Y_\infty\times G(\A_f) \>\>\>\text{ and }\>\>\> X^\A \coloneqq G(\Q)\bs Y^\A,\] where $\A_f$ is the ring of finite adeles of $\Q$. The (sheaf) cohomology $H^*(X^\A;\C)$ can be identified with the colimit $\mathrm{colim}\, H^*(X^\A/K_f;\C)$, where $K_f\sbs G(\A_f)$ ranges over open compact subgroups. Each $X^\A/K_f$ is a finite disjoint union $\bigsqcup_i \Ga_i\bs Y_\infty$ with $\Ga_i< G(\Q)$ an arithmetic lattice.

\begin{deftn}The \emph{automorphic cohomology} of $G$ is given by
\begin{equation}\label{eqn:colimit}H^*(G;\C) \coloneqq \mathrm{colim}\, H^*(X^\A/K_f;\C).\end{equation}\end{deftn}

In this framework, there is a map \cite[pg.\ 1062]{grobner} 
\begin{equation}\label{eqn:adelic}\Psi \colon H^*(\mf g,K;\C)\lra H^*(G;\C)\,\end{equation}
where $H^*(\mf g,K;\C)$ is relative Lie algebra cohomology with trivial coefficients. The construction of the Matsushima homomorphism \eqref{eqn:matsushima} passes through the isomorphism $H^*(X_u;\C)\cong H^*(\mf g,K;\C)$  \cite[\S 4]{okun}, \cite[\S 10]{borel_cohoarith}. In the proof of Proposition \ref{prop:adelic-injective} we will explain that the Matsushima homomorphism picks out the contribution of the trivial representation to the automorphic cohomology. In particular, it fits in a commutative diagram
\begin{equation}\label{eqn:matsushima-diagram}\begin{tikzcd} H^*(\mf g,K;\C) \rar{\Psi} & H^*(G;\C) \\
H^*(X_u;\C) \uar{\cong} \rar{\mu} & H^*(B\Ga;\C) \uar,\end{tikzcd}\end{equation}
with right vertical induced by the map $B\Ga \to \Gamma\backslash Y_\infty \hookrightarrow \bigsqcup_i \Gamma_i\backslash Y_\infty = X^\A/K_f$ for suitable $K_f$.

We will see that Theorem \ref{thm:franke} follows the following result regarding the homomorphism (\ref{eqn:adelic}).

\begin{prop}\label{prop:adelic-injective}The homomorphism (\ref{eqn:adelic}) is injective in degrees $*\le \min_R\dim N_R$.\end{prop}

The proposition follows from \cite{franke,grobner}. There is a small amount of work needed to translate the results of these papers to our setting. 

\begin{proof}
First we explain a weaker statement: (\ref{eqn:adelic}) is injective in degrees $*<\frac{1}{2}\min_R\dim N_R$.\footnote{Although including this argument is not strictly necessary, this statement is already sufficient for Theorem \ref{thm:nielsen} and the argument illustrates the connection between the Matsushima homomorphism and automorphic forms.} This is proved directly in \cite{grobner}, building on \cite{frankeHarm, FrankeSchwermer,franke}. We explain only what is needed for our argument, and refer to \cite{grobner} for more details. The cohomology $H^*(G;\C)$ can be identified with relative Lie algebra cohomology $H^*(G;\C)=H^*(\mf g,K;\ca A(G))$, where $\ca A(G)$ is a space of automorphic forms \cite[Introduction]{grobner}. (Comparing with Grobner's notation, we remark that since $G$ is semisimple in our case, the quotient $\mf m_G$ in \cite{grobner} is just the Lie algebra $\mf g$; furthermore, since we are only interested in the trivial representation $E=\C$, we will write $\ca A(G)$ instead of $\ca A_{\ca J}(G)$.) 

By \cite{frankeHarm} and \cite{FrankeSchwermer}, there is a decomposition 
\[\ca A(G)=\bigoplus_{\{P\}}\bigoplus_{\phi_P}\ca A_{\{P\},\phi_P}(G),\>\text{ and hence also } \>H^*(G;\C)=\bigoplus_{\{P\}}\bigoplus_{\phi_P}H^*\big(\mf g,K;\ca A_{\{P\},\phi_p}(G)\big),\]
where $\{P\}$ ranges over (associate classes) of $\Q$-parabolic subgroups and $\phi_P$ ranges over (associate classes) of cuspidal automorphic representations of the Levi subgroups of elements of $\{P\}$; see \cite[\S2]{grobner}. The summands of $\ca A(G)$ corresponding to $P\neq G$ is denoted $\ca A_{\Eis}(G)$, and the corresponding subspace $H^*_{\Eis}(G;\C)\sbs H^*(G;\C)$ is called the \emph{Eisenstein cohomology}. The constant functions span a trivial sub-representation $1_{G(\bb A)}\subset\ca A_{\Eis}(G)$. This defines a map $H^*(\mf g,K;\C)\ra H^*(G;\C)$, which is precisely the map (\ref{eqn:adelic}). Necessarily  $1_{G(\bb A)}$ is contained in a unique summand $\ca A_{\{P\},\phi_P}(G)$. Then by \cite[Cor.\ 17]{grobner}, the induced map $H^*(\mf g,K;\C)\ra H^*\big(\mf g,K;\ca A_{\{P\},\phi_P}(G)\big)$ is injective in a range $0\le *< q_{\text{res}}$, where the constant $q_{\text{res}}=q_{\text{res}}(P,\phi_P)$ is defined in \cite[\S6]{grobner}. As discussed in \cite[\S7.4]{grobner}, since we are working with the trivial representation, $q_{\text{res}}$ is equal to the constant $q_{\max}=\frac{1}{2}\min_R\dim N_R$ defined in \cite[\S7.1]{grobner} (note that in our case $G$ is defined over $\Q$, which as only one place, so the sum in Grobner's definition of $q_{\max}$ has only one term). 

Next we explain how to deduce from \cite{franke} that (\ref{eqn:adelic}) is injective for $*\le\min_R\dim N_R$. 

We define $H^*_c(G;\C)$ to be the colimit of the compactly supported cohomology groups $H^*_c(X^\A/K_f;\C)$. Using Poincar\'e duality for each of the symmetric spaces $\Gamma_i \slash Y_\infty$, the map $\Psi:H^*(X_u;\C)\cong H^*(\mf g,K;\C)\ra H^*(G;\C)$ has a dual map 
\[\Psi' \colon H_c^*(G;\C)\lra H^*(X_u;\C)\] 
on compactly supported cohomology. Then $\ker(\Psi)=\im(\Psi')^\perp$, where the orthogonal complement is with respect to the cup product $\smile$ on $H^*(X_u;\C)$. Franke \cite{franke} gives a precise description of $\im(\Psi')$. To describe it, fix a minimal parabolic $P_0<G$, and consider a parabolic subgroup $R\supset P_0$. Write $R=MAN$ for the Langlands decomposition, where $M$ is semi-simple, $A$ is abelian, and $N$ is unipotent. When we vary $R$, we write $M_R, N_R$ for emphasis. The compact dual symmetric space of $M$, denoted $X_M$, embeds in $X_u$. Franke proves that $\im(\Psi')=\ker(\Phi)$, where 
\[\Phi \colon H^*(X_u;\C)\lra \prod H^*(X_M;\C)\]
is the map induced by the inclusions $X_M\hra X_u$, ranging over $R=MAN$ maximal parabolic subgroups (maximal is equivalent to $\dim A=1$). See \cite[(7.2) pg.\ 59]{franke} and also \cite[\S2,3]{speh-venkataramana}. Thus we have
\[\ker(\Psi)=\ker(\Phi)^\perp.\]

To show that $\Psi$ is injective in low degrees, we use the following observation: if $V^\perp\sbs\ker(\Phi)$ for some subspace $V\sbs H^*(X_u;\C)$, then $\ker(\Psi)=\ker(\Phi)^\perp\sbs V$. This implies that $\Psi$ is injective in degrees $*<\min_{0 \neq v\in V}\deg(v)$.

Fix $R$, and consider the inclusion $i:X_M\ra X_u$. For $k\ge1$, observe that $a\in H^k(X_u;\C)$ belongs to $\ker(i^*)$ if and only if $a\smile\PD(i_*(z))=0$ for every $z\in H_k(X_M;\C)$. Here $\PD(\cdot)$ denotes Poincar\'e duality. Then $V^{\perp}\sbs\ker(\Phi)$, where $V\sbs H^*(X_u;\C)$ is defined as the image of
\[\bigoplus_{R}\bigoplus_{k\ge1}H_{k}(X_{M_R};\C)\xra{i_*} H_*(X_u;\C)\xra{\PD} H^*(X_u;\C),\]
where $\bigoplus_R$ ranges over maximal parabolic subgroups containing $P_0$ as before. Observe that classes in $H_*(X_M;\C)$ of low dimension map to classes in $H^*(X_u;\C)$ of low \emph{codimension}. Thus if $v\in V$, then $\deg(v)\ge\dim X_u-\dim X_M$ for each $M$. Therefore, $\Psi$ is injective in degrees $*<\min_{R}\big(\dim X_u-\dim X_{M_R}\big)$. 

Finally, we show the minimum codimension of $X_M\sbs X_u$ is equal to $1+\min_R\dim R$. This follows quickly from the Iwasawa decomposition for a semi-simple Lie group and Langlands decompositions for a parabolic subgroup. By the Iwasawa decomposition, we can write $G=KAN$, where $K$ is maximal compact. For our maximal parabolic $R$, we have $R=MA_RN_R$, and furthermore, since $M$ is semisimple, it has an Iwasawa decomposition $M=K_MA_MN_M$. Observe that $\dim X_u=\dim AN$, $\dim X_M=\dim A_MN_M$, and $\dim AN=\dim A_MN_M+\dim A_RN_R$. Then
\[\dim X_u-\dim X_M=\dim A_RN_R=1+\dim N_R.\]
This completes the proof.
%
\end{proof}

\begin{proof}[Proof of Theorem \ref{thm:franke}] 
For any $x\in H^*(\mf g,K;\C)$ in the given range, by the injectivity of $\Psi$ and the description \eqref{eqn:colimit} of $H^*(G;\C)$ as a colimit, there is an arithmetic lattice $\Ga'<G(\Q)$ so that $\Psi(x)$ is in the image of $H^*(\Ga';\C)\ra H^*(G;\C)$, as in (\ref{eqn:matsushima-diagram}). By transfer, the same is true for any further finite-index subgroup of $\Ga'$. Then since $H^*(\mf g,K;\C)$ is degree-wise finite-dimensional, in the desired range \eqref{eqn:adelic} 
provides an injective map $H^*(\mf g,K;\C) \to H^*(\Ga';\C)$ for some arithmetic lattice $\Gamma' \leq G(\Q)$. Any arithmetic lattice $\Ga \leq G(\Q)$ is commensurable to $\Ga'$, and hence $\Ga$ and $\Ga'$ have a common finite index subgroup $\Ga''$. Consider the commutative diagram
\[\begin{tikzcd} & H^*(B\Gamma';\C)  \ar{rd} & \\[-17pt]
H^*(\mf g,K;\C) \ar{rd} \ar{ru} & & H^*(B\Gamma'';\C). \\[-17pt]
& H^*(B\Gamma;\C)\ar{ru}   & \end{tikzcd} \]
By a transfer argument the top composition is injective in the desired range, and hence so is $H^*(\mf g,K;\C) \to H^*(B\Gamma;\C)$, proving that \eqref{eqn:adelic} and hence \eqref{eqn:matsushima} is injective in the desired range.\end{proof}

In the remainder of this section we compute Franke's constant $\min_R\dim N_R$ for $G=\SO(p,q)$. We also compute Franke's constant for $G=\Sp_{2g}$ and $G=\Sl_n$, since these are examples of common interest.

\subsection{Special orthogonal groups} 
Fix $1\le p\le q$, set $d=q-p$, and consider the algebraic group 
\[\SO(B) \coloneqq \{g\in\SL_{p+q} \mid g^tBg=B\},\]
where $B$ is the $(p+q) \times (p+q)$-matrix given by
\[B=\left(\begin{array}{ccc}
I_p&0\\
0&-I_q
\end{array}\right).\]

The associated compact dual symmetric space is $X_u = \SO(p+q)/(\SO(p) \times \SO(q))$, whose cohomology $H^*(X_u;\C)$ can be computed using \cite[Theorem 8.2]{mccleary}.

\begin{prop}\label{prop:franke-SO}
	Fix a finite-index subgroup $\Ga \leq \SO(B;\Z)$. Then the Matsushima homomorphism $H^*(X_u;\C)\ra H^*(B\Ga;\C)$ is injective in degrees $*\le p+q-2$. 
\end{prop}

\begin{proof}
	By the preceding discussion, it suffices to prove 
	\[\min_R\dim N_R=p+q-2,\] where $R$ ranges over a maximal parabolic subgroups over $\Q$, and $N_R$ is the unipotent radical. Parabolic subgroups of $\SO(B;\R)$ are stabilizers of isotropic flags in $(\R^{p+q},B)$. A maximal parabolic subgroup is specified by a single non-trivial isotropic subspace. Let $e_1,\ldots,e_p,f_1,\ldots,f_q$ be the standard basis for $\R^{p+q}$ (whose Gram matrix is $B$). Denoting $u_i=e_i+f_i$, 
	let $R_k<\SO(B;\R)$ be the stabilizer of $W=\R\{u_1,\ldots,u_k\}$ for $1\le k\le p$. Every maximal parabolic subgroup is conjugate to some $R_k$. 
	
	Fix $1\le k\le p$. An element of $R_k$ preserves the flag $0\sbs W\sbs W^\perp\sbs\R^{p+q}$. The unipotent radical $N_k\sbs R_k$ is the subgroup that acts trivially on each of the quotients $W/0$, $W^\perp/W$, $\R^{p+q}/W^\perp$. To determine $\dim N_k$, denote $v_{i}=e_i-f_i$ for $1\le i\le p$, and work in the ordered basis 
	\[u_1,\ldots,u_k, u_{k+1},\ldots,u_p, f_{p+1},\ldots,f_q,v_{k+1},\ldots,v_p,v_1,\ldots,v_k.\] 
	Then $g\in N_k$ can be expressed as a block matrix 
	\[g=\left(\begin{array}{ccc}
	I_k&y&z\\
	0&I_{p+q-2k}&x\\
	0&0&I_k
	\end{array}\right),\] 
	where $y=-x^tQ$ and $z+z^t=x^tQx$ and $Q$ is the $(p+q-2k)\times(p+q-2k)$ matrix
	\[Q=\left(\begin{array}{ccc}
	0&0&I_{p-k}\\
	0&I_{q-p}&0\\
	I_{p-k}&0&0
	\end{array}\right).\] 
	The homomorphism $N_k\ni g\mapsto x\in\R^{k(p+q-2k)}$ has kernel the space of skew-symmetric matrices $z^t=-z$, so $\dim N_k=k(p+q-2k)+\fr{k(k-1)}{2}$. For $1\le k\le p$, this number is smallest when $k=1$, which gives the constant claimed in the theorem. 
\end{proof}

\begin{proof}[Proof of Proposition \ref{prop:frankeEin}] Since $M$, the intersection form of the $K3$ manifold, is equivalent to $B$ over $\R$ with $p=3$ and $q =19$. Thus when we apply Theorem \ref{thm:franke}, the same estimates as in Proposition \ref{prop:franke-SO} holds. Thus the map $H^*(X_u;\C)\ra H^*(B\Ga_{\Ein};\C) \to H^*(B\Ga;\C)$ is injective for $*\le 20$ and hence so is $H^*(X_u;\C)\ra H^*(B\Ga_{\Ein};\C)$.\end{proof}

\subsection{Symplectic groups} We next specialize Theorem \ref{thm:franke} to finite index subgroups of symplectic groups. Take $G = \Sp_{2n}$ to be the algebraic group defined by 
\[\Sp_{2n} \coloneqq \{g\in\Sl_{2n} \mid g^tJ_ng=J_n\},\]
where $J_n$ is the $2n\times 2n$ matrix given by 
\[J_n \coloneqq \left(\begin{array}{cc}0&I_n\\-I_n&0\end{array}\right).\]
The associated compact dual symmetric space is $X_u = \Sp(n)/\U(n)$, whose cohomology in the range below is the polynomial algebra on generators $c_1,c_3,c_5,\ldots$ with $|c_i| = 2i$. 

\begin{prop}\label{prop:franke-Sp}
	For any finite-index subgroup $\Ga \leq \Sp_{2n}(\Z)$ the Matsushima homomorphism $H^*(X_u;\C)\ra H^*(B\Ga;\C)$ is injective in degrees $*\le 2n-1$.
\end{prop}

\begin{proof}[Proof of Proposition \ref{prop:franke-Sp}]
	The proof follows from Theorem \ref{thm:franke} similar to Proposition \ref{prop:franke-SO}. Let $e_1,\ldots,e_n$, $f_1,\ldots,f_n$ be the standard symplectic basis for $\R^{2n}$. Let $R_k$ be the maximal parabolic subgroup of $\Sp_{2n}$ defined as the stabilizer of $W=\R\{e_1,\ldots,e_k\}$ for $1\le k\le n$. Working in the basis $e_1,\ldots,e_k, e_{k+1},\ldots,e_n, f_{k+1},\ldots,f_n, f_1,\ldots, f_k$, an element of the unipotent radical $N_k$ can be expressed as a block matrix 
	\[g=\left(\begin{array}{ccc}
	I_k&y&z\\
	0&I_{2n-2k}&x\\
	0&0&I_k
	\end{array}\right),\]
	where $y=x^tJ'$ and $z-z^t=y^tJ'y$ and $J'=J_{n-k}$. It follows that $\dim N_k=2k(n-k)+k+\fr{k(k-1)}{2}$. For $1\le k\le n$, this number is smallest when $k=1$. 
\end{proof}

\subsubsection{The tautological ring of $\ca A_g$}
Let $\ca A_g$ denote the moduli space of principally polarized abelian varieties. The tautological ring $R^*_{CH}(\ca A_g) \subset CH^*(\ca A_g;\Q)$ in the Chow ring is the subalgebra generated by the $\lambda$-classes $\lambda_i \in CH^{2i}(\ca A_g;\Q)$, the Chern classes of the Hodge bundle (the $2g$-dimensional vector bundle given at an abelian variety $X \in \ca A_g$ by the tangent space to its identity element). Van der Geer proved it has a $\Q$-basis given by the monomials $\lambda_1^{a_1} \lambda_2^{a_2} \cdots \lambda_{g-1}^{a_{g-1}}$ with $a_i \in \{0,1\}$ \cite[Theorem (1.5)]{vanderGeer} \cite[\S 4]{vanderGeerCohomology}. As for Chow groups, there is a tautological ring $R^*_H(\ca A_g) \subset H^*(\ca A_g;\Q)$ in rational cohomology defined as the subalgebra generated by the $\lambda$-classes. In the literature it is claimed van der Geer's computation also holds in cohomology, but no reference for this is known to the authors. We provide a proof below:

\begin{thm}The tautological ring $R^*_H(\ca A_g) \subset H^*(\ca A_g;\Q)$ has a $\Q$-basis given by the monomials $\lambda_1^{a_1} \lambda_2^{a_2} \cdots \lambda_{g-1}^{a_{g-1}}$ with $a_i \in \{0,1\}$.\end{thm} 

\begin{proof}$R^*_{CH}(\ca A_g)$ surjects onto $R^*_H(\ca A_g)$, so it suffices to prove they have the same dimension. 
	
The space $\ca A_g$ is the quotient of the contractible Siegel upper half space $\mathbb{H}_g$ by $\Sp_{2g}(\Z)$. This action has finite stabilizers, so there is an isomorphism $H^*(\ca A_g;\Q) \cong H^*(\Sp_{2g}(\Z);\Q)$. Under the isomorphism  $H^*(\ca A_g;\C) \cong H^*(\Sp_{2g}(\Z);\C)$, $R^*_H(\ca A_g) \otimes \C$ is exactly the image of the Matsushima homomorphism \cite[\S 10]{vanderGeerCohomology}.  In \cite[Section 4]{speh-venkataramana}, Speh and Venkataramana prove that the kernel of the Matsushima homomorphism is the orthogonal complement of the ideal $(u_g)$ in
\[H^*(X_u;\Q) \cong \frac{\Q[u_1,\cdots,u_g]}{((1+u_1+u_2+\cdots+u_g)(1-u_1+u_2-\cdots+(-1)^g u_g)-1)}.\]
This is \cite[Lemma 8]{speh-venkataramana}, combined with the description of $H^*(X_u;\Q)$ in \cite[\S 1]{vanderGeer}. The latter also proves there is an isomorphism $H^*(X_u;\Q) \cong R^*_{CH}(\ca A_{g+1})$ identifying $u_i$ with $\lambda_i$. In particular, from the basis given above we see that the kernel of the Matsushima homomorphism is spanned by the monomials $u_1^{\epsilon_1} u_2^{\epsilon_2} \cdots u_{g-1}^{\epsilon_{g-1}} u_g$ with $\epsilon_i \in \{0,1\}$. Thus the image of the Matsushima homomorphism has the same dimension as $R^*_{CH}(\ca A_g)$, and the result follows.\end{proof}

Observe this result in particular describes the image of the Matsushima homomorphism in $H^*(B\Ga;\C)$ for finite-index subgroups $\Ga \subset \Sp_{2g}(\Z)$.

\subsection{Special linear groups} Finally, we specialize Theorem \ref{thm:franke} to finite-index subgroups of special linear groups. Now we have $G = \SL_n$ and $X_u = \SU(n)/\SO(n)$, whose cohomology in the range below is the exterior algebra on generators $\bar{c}_3,\bar{c}_5,\bar{c}_7,\ldots$ with $|\bar{c}_i| = 2i-1$.

\begin{prop}\label{thm:franke-SL}
	For any finite-index subgroup $\Ga \leq \SL_n(\Z)$ the Matsushima homomorphism $H^*(X_u;\C)\ra H^*(B\Ga;\C)$ is injective in degrees $*\le n-1$.
\end{prop}

The proof is similar to the proof of Propositions \ref{prop:franke-SO} and \ref{prop:franke-Sp}, but simpler; one identifies the maximal parabolic subgroups over $\Q$ as the stabilizers of a non-trivial subspace $W$ and observes that the stabilizers of 1-dimensional subspaces have the smallest unipotent radical, of dimension $n-1$.

\subsubsection{A result announced by Lee} In \cite[Theorem 1]{Lee}, Lee announced a result which in particular implies that the range in Proposition \ref{thm:franke-SL} can be improved to $* \le 2n-3$. His result can be deduced from page 61 of \cite{franke}, where Franke describes the kernel of the Matsushima homomorphism for finite index subgroups of $\SL_n(\ca O_K)$, with $\ca O_K$ the ring of integers in a number field $K$:

\begin{thm}\label{thm:sln-lee} For any finite-index subgroup $\Ga \leq \SL_n(\Z)$, the image of the Matsushima homomorphism $H^*(X_u;\C)\ra H^*(B\Ga;\C)$ is an exterior algebra on the classes $\bar{c}_3,\cdots,\bar{c}_{n-1}$ with $|\bar{c}_i| = 2i-1$ when $n$ is odd, and an exterior algebra on the classes $\bar{c}_3,\cdots,\bar{c}_{n-3}$ when $n$ is even.\end{thm}

\begin{proof}The cohomology of compact dual $X_u$ for $\SL_n(\Z)$ is given by the following exterior algebras:
	\[H^*(X_u;\Q) = \begin{cases} \Lambda(\bar{c}_3,\cdots,\bar{c}_{n}) & \text{if $n$ is odd}\\
	\Lambda(\bar{c}_3,\cdots,\bar{c}_{n-1},e) & \text{if $n$ is even,}
	\end{cases} \]
	with $|\bar{c}_i| = 2i-1$ and $|e| = n$. According to page 61 of \cite{franke}, when $n$ is odd the kernel of Matsushima homomorphism is the ideal generated by $\bar{c}_n$, and when $n$ is even it is the ideal generated by $\bar{c}_{n-1}$ and $e$.
\end{proof}

\begin{rmk}Theorem \ref{thm:sln-lee} resolves a question in \cite[Remark 7.5]{elbaz-vincent}; the Borel class $\bar{c}_3$ is non-zero in $H^5(B\SL_n(\Z);\Q)$ for $n \geq 5$, and the Borel class $\bar{c}_5$ is non-zero in $H^9(B\SL_n(\Z);Q)$ for $n \geq 7$. Similarly $\bar{c}_3\bar{c}_5$ is non-zero in $H^{14}(B\SL_n(\Z);\Q)$ for $n \geq 7$. Curiously, the non-zero class they find in $H^9(B\SL_6(\Z);\Q)$ is \emph{not} stable.\end{rmk}

\section{Moduli of Einstein metrics} To apply our knowledge of the cohomology of arithmetic groups, we use the global Torelli theorem to study the moduli space $\ca M_{\Ein}$ of Einstein metrics on the $K3$ manifold. Following \cite[\S 4]{giansiracusa}, for us this shall mean the homotopy quotient 
\[\ca M_{\Ein} \coloneqq \ca T_{\Ein} \sslash \Ga_{\Ein}\]
of a moduli space $\ca T_{\Ein}$ of marked Einstein metrics by the subgroup $\Ga_{\Ein} \leq \Ga_K$. The space $\ca T_{\Ein}$ admits a description as a hyperplane complement, but we only use a pair of consequences of this.

Fix a finite-index subgroup $\Ga' \leq \Ga_K$, and assume $\Ga'$ is contained in $\Ga_{\Ein}$. Equivalently, one may assume it is contained in the identity component of $\OO(3,19)$. We introduce the notation $\Mod_{\Ein} \coloneqq \alpha^{-1}(\Ga_{\Ein})$ and $\Mod' \coloneqq \alpha^{-1}(\Ga')$.

\begin{prop}\label{prop:torelli} The homomorphism $H^*(B\Ga';\C)\ra H^*(B\Mod';\C)$ is injective for any $\Ga' \leq  \Ga_K$.
\end{prop}

\begin{proof}We will first prove that the surjection $\Mod(K)\ra\Ga_K$ splits over $\Ga_{\Ein}$ by Giansiracusa's work: there is a map 
	\begin{equation}\label{eqn:torelli}
	e \colon \ca M_{\Ein}\lra B\Diff(K)\lra B\Mod(K)\lra B\Ga_K.\end{equation} The induced homomorphism $\pi_1(\ca M_{\Ein})\ra\Ga_K$ is injective with image $\Ga_{\Ein}$ by the global Torelli theorem \cite[\S 4-5]{giansiracusa}. Thus, $\Mod(K)\ra\Ga_K$ splits over $\Ga_{\Ein}$. This proves the case $\Ga' = \Ga_{\Ein}$; for $\Ga' \subset \Ga_{\Ein}$ one restricts the splitting to $\Ga'$. 
	
	If $\Ga' \not \subset \Ga' \cap \Ga_{\Ein}$, then $\Ga' \cap \Ga_{\Ein}$ has index 2 in $\Ga'$ and similarly $\Mod' \cap \Mod_{\Ein}$ has index 2 in $\Mod'$. Thus the injective homomorphism $H^*(B(\Ga' \cap \Ga_{\Ein});\C) \to H^*(B(\Mod' \cap \Mod_{\Ein});\C)$ is one of representations of $\Z/2 \cong \Ga'/(\Ga' \cap \Ga_{\Ein}) = \Mod'/(\Mod' \cap \Mod_{\Ein})$, and we can identify $H^*(B\Ga';\C)\ra H^*(B\Mod';\C)$ with the induced map on $\Z/2$-invariants. As taking $\Z/2$-invariants preserves injective maps, the proposition follows.\end{proof} 

To prove Theorem \ref{thm:hirzebruch} we must prove that $p^* x_4 \neq 0\in H^4(B\Diff(K);\Q)$. To do so, it suffices to prove that is non-zero when pulled back to $\ca M_{\Ein}$:

\begin{prop}For the map $e$ defined in $(\ref{eqn:torelli})$, $e^* x_4 \neq 0 \in H^4(\ca M_{\Ein};\Q)$.\end{prop}

\begin{proof}We will prove that $e^* \colon H^4(B\Ga_K;\Q) \to H^4(\ca M_{\Ein},\Q)$ is injective. In \cite[\S 5]{giansiracusa}, one finds a description of the Serre spectral sequence for the fibration sequence
	\[\ca T_{\Ein} \lra \ca M_{\Ein} = \ca T_{\Ein} \sslash \Ga_{\Ein} \lra B\Ga_{\Ein}.\]
	Its $E^2$-page is given by
	\[E^2_{p,q} = \begin{cases} 0 & \text{if $q$ is odd,} \\
	\prod_{\sigma \in \Delta_{q/2}/\Ga_{\Ein}} H^p(B \Stab(\sigma);\Q) & \text{if $q$ is even.}\end{cases}\]
	The description of $\De_{q/2}/\Ga_{\Ein}$ is not important here, as we shall only use the rows $0 \leq q \leq 3$. Of these, the following are non-zero: for $q=0$ we get $H^p(B\Ga_K;\Q)$, and for $q=2$ we get a product of the cohomology groups of groups $\Ga$ commensurable with $\OO(2,19;\Z)$ or $\OO(3,18;\Z)$. For such groups $H^1(\Ga;\Q)$ vanishes \cite[Corollary 7.6.17]{margulis}, and thus there can not be any non-zero differential into the entry $E^2_{4,0}$.	
\end{proof}

\section{Nielsen realization}\label{sec:nielsen}
We now deduce Theorem \ref{thm:nielsen} from either Proposition \ref{prop:hirzebruchimprov} or \ref{prop:frankeEin}. The argument in fact shows that $\Diff(K) \ra \Mod(K)$ does not split over any finite index subgroup of $\Mod(K)$.

\begin{proof}[Proof of Theorem \ref{thm:nielsen}] We will show that $\Diff(K) \ra \Mod(K)$ does not split by contradiction, so we assume there is a splitting $s \colon \Mod(K) \to \Diff(K)$, which necessarily factors over the discrete group $\Diff(K)^\delta$ as 
	\[\Mod(K) \overset{s^\delta} \lra \Diff(K)^\delta \overset{p^\delta} \lra \Diff(K).\]
	
	Observe that $x_8 \in H^8(B\Ga_K;\Q)$ is non-zero; either one pulls back to $B\Diff(K)$ and uses Proposition \ref{prop:hirzebruchimprov} and Lemma \ref{lem:atiyah}, or one pulls back to $B\Gamma_{\Ein}$ and uses Proposition \ref{prop:frankeEin}. By Proposition \ref{prop:torelli} its pullback to $H^8(B\Mod(K);\Q)$, which we denote by $c$, is also non-zero. Its pullback under
\[B\Mod(K) \overset{s^\delta} \lra B\Diff(K)^\delta \overset{p^\delta} \lra B\Diff(K) \overset{p} \lra  B\Mod(K)\]
is $c$ and hence non-zero. By Section \ref{sec:mmm} we get $p^* c = \kappa_{\ca L_3}$ and we claim that $(p^\delta)^* \kappa_{\ca L_3} \in H^{8}(B\Diff(K)^\delta)$ vanishes. This would contradict $c \neq 0$ and finish the proof. To prove the claim, we use that $B\Diff(K)^\delta$ classifies flat $K$-bundles, i.e.~bundles with a foliation transverse to the fibers and of codimension 4. The normal bundle to this foliation is isomorphic to the vertical tangent bundle, and by the Bott vanishing theorem its Pontryagin ring vanishes in degrees $>8$ \cite{Bott}. In particular the class $\ca L_3$ of degree 12 vanishes.\end{proof}
	
	\begin{rmk}
		The idea of using Bott vanishing to obstruct Nielsen realization originates in Morita's work \cite[Thm.\ 8.1]{morita_nonlifting}.
	\end{rmk}

\bibliographystyle{alpha}
\bibliography{K3nielsen}

\end{document}